\documentclass{article}
\usepackage{amsmath}
\usepackage{amssymb}
\usepackage{amsfonts}
\usepackage{float}
\usepackage{times}
\usepackage{graphicx}
\usepackage{color}
\usepackage{url}

\newtheorem{theorem}{Theorem}

\newtheorem{corollary}[theorem]{Corollary}

\newenvironment{proof}[1][Proof]{\noindent\textbf{#1.} }{\ \rule{0.5em}{0.5em}}
\input{tcilatex}
\begin{document}

\title{Minimal Sum of Powered Distances from the Sides of a Triangle}
\author{Elias Abboud \\
{\small The Academic Arab Institute, Faculty of Education, }\\
{\small Beit Berl College, Doar Beit Berl 44905, Israel}\\
{\small Email: eabboud@beitberl.ac.il\ }}
\maketitle

\section*{Introduction}

The Karush-Kuhn-Tucker (KKT) conditions are used to solve inequality
constrained nonlinear optimization problems. The systematic treatment of
inequality constraints was published in 1961 by Kuhn and Tucker \cite{KT}
and later it was found that the elements of the theory were contained in the
1939 unpublished M.Sci Dissertation of W. Karush at the University of
Chicago (see \cite[p. 355]{LY}). One type of inequality constrained
minimization for a nonlinear programming problem is of the form%
\begin{gather}
\text{minimize\ \ }f(\mathbf{x})  \notag \\
\text{s.t. }\mathbf{g(x)}\geq \mathbf{0}  \label{E1} \\
\mathbf{x}\in 
\Omega
\subseteq \mathbb{R}^{n},  \notag
\end{gather}%
where $f$ is a $C^{1}$ real-valued function and the feasible set $%
\Omega
$ is a subset of $\mathbb{R}^{n}$.

A point $\mathbf{x}\in 
\Omega
$ that satisfies $\mathbf{g(x)=(}g_{1}\mathbf{(x),...,}g_{m}\mathbf{(x))}%
\geq \mathbf{0}$ is said to be \textit{feasible}. An inequality constraint $%
g_{i}\mathbf{(x)}\geq 0$ is said to be \textit{active} at a feasible point $%
\mathbf{x}$ if $g_{i}\mathbf{(x)}=0$ and \textit{inactive} at $\mathbf{x}$
if $g_{i}\mathbf{(x)}>0$. A point $\mathbf{x}^{\ast }$ is said to be a 
\textit{regular point} of the constraints $g_{i}\mathbf{(x)}\geq 0,1\leq
i\leq m,$ if the gradient vectors $\nabla g_{i}\mathbf{(\mathbf{x}^{\ast })}%
,1\leq i\leq m,$ are linearly independent.

The KKT conditions are necessary conditions for a relative minimum (see \cite%
[p.340]{LY}): Let $\mathbf{x}^{\ast }$ be a relative minimum point for
problem (\ref{E1}) and suppose $\mathbf{x}^{\ast }$ is a regular point for
the constraints. Then there is a vector $\mathbf{\lambda }\in \mathbb{R}^{m},%
\mathbf{\lambda }\geq \mathbf{0,}$ of Lagrange multipliers such that $f(%
\mathbf{x}^{\ast })-\mathbf{\lambda }^{T}\nabla \mathbf{g(\mathbf{x}^{\ast })%
}=\mathbf{0}$ and $\mathbf{\lambda }^{T}\nabla \mathbf{g(\mathbf{x}^{\ast })}%
=0.$

For further reading on the KKT conditions we refer the readers to \cite{Ber}
and \cite{LY}. In this paper we shall apply the KKT conditions to find the
minimal sum of powered distances from the sides of an arbitrary triangle.
The special case $n=1$ was reviewed in \cite{Abb1} and \cite{Abb}. By
defining a suitable linear programming problem, it was concluded that the
sum of distances attains its minimum at the vertex through which the
smallest altitude of the triangle passes and this minimum equals the length
of the smallest altitude. In \cite{Abb}, the case $n=2$ was established for
isosceles triangles with vertices $A(0,a),B(-b,0)$\ and $C(b,0)$:\ The
minimal sum of squared distances from the sides of the isosceles triangle is 
$\frac{2a^{2}b^{2}}{a^{2}+3b^{2}}$\ attained at the point $(0,\frac{2ab^{2}}{%
a^{2}+3b^{2}}),$\ inside the triangle.

\subsection*{Formulation of the general problem}

Let $\Delta $ be the closed triangle, which includes both boundary and inner
points, with vertices $A(0,a),B(-b,0)$\textit{\ }and\textit{\ }$C(c,0),$
where $a>0,b>0$ and $\ c>0$. If the triangle is obtuse then we take $A(0,a)$
to be the vertex of the obtuse angle. Obviously, the side $AB$ lies on the
line $ax-by+ab=0$ and the side $AC$ lies on the line $ax+cy-ac=0$. Given a
point $P(x,y)$ in the plane, let $d_{1},d_{2}$ and $d_{3}$ be the distances
of $P$ from the sides $AB$, $AC$ and $BC,$ respectively. Evidently, these
distances satisfy the following equations:

\begin{eqnarray}
d_{1} &=&\left\vert \frac{ax-by+ab}{\sqrt{a^{2}+b^{2}}}\right\vert ,  \notag
\\
d_{2} &=&\left\vert \frac{-ax-cy+ac}{\sqrt{a^{2}+c^{2}}}\right\vert ,
\label{eq1} \\
d_{3} &=&\left\vert y\right\vert .  \notag
\end{eqnarray}

Notice that the absolute value can be omitted if the point $P(x,y)$ lies in $%
\Delta ,$ since $ax-by+ab\geq 0$, $-ax-cy+ac\geq 0$ and $y\geq 0.$ For any
integer $n,$ $n\geq 1,$ and any point $P(x,y)$ in the plane define the 
\textit{sum of powered distances function} 
\begin{equation}
F(x,y)=\left\vert \frac{ax-by+ab}{\sqrt{a^{2}+b^{2}}}\right\vert
^{n}+\left\vert \frac{-ax-cy+ac}{\sqrt{a^{2}+c^{2}}}\right\vert
^{n}+\left\vert y\right\vert ^{n}.  \label{eq1.4}
\end{equation}

The following observation is easy to verify so we omit the proof: \textit{%
For any point }$P(x,y)$\textit{\ outside the triangle }$\Delta $\textit{\
there exist a point }$P^{\prime }(x^{\prime },y^{\prime })$\textit{\ on the
boundary of }$\Delta $\textit{\ such that }$F(x^{\prime },y^{\prime })\leq
F(x,y).$\textit{\ }

Thus, we may restrict our feasible region to the closed triangle $\Delta $
and consider the nonlinear constrained optimization problem: 
\begin{eqnarray}
&&\text{minimize }F(x,y)  \notag \\
&&\text{s.t. \ }  \label{eq1.5} \\
&&\left\{ 
\begin{array}{c}
ax-by+ab\geq 0 \\ 
-ax-cy+ac\geq 0 \\ 
y\geq 0%
\end{array}%
\right. .  \notag
\end{eqnarray}

For any integer $n,$ $n\geq 2,$ denote 
\begin{eqnarray}
p &=&\sqrt{a^{2}+b^{2}}\text{, }q=\sqrt{a^{2}+c^{2}},  \notag \\
t &=&\sqrt[n-1]{\frac{p}{q}}\text{, }r=\sqrt[n-1]{\frac{b+c}{q}}\text{, }%
\lambda =q+br+cr+pt.  \label{eq2}
\end{eqnarray}

\section*{Minimal sum of powered distances}

\ We shall prove that the minimum of \ $%
F(x,y)=d_{1}^{n}+d_{2}^{n}+d_{3}^{n},n>1$, in the closed triangle is
attained at a unique point inside the triangle. The proof contains technical
computations which can be carried on using any Computer Algebra System.

\begin{theorem}
\label{mim inside triangle}For any integer $n\geq 2,$ the minimum in the
closed triangle $\Delta $ with vertices $A(0,a),B(-b,0)$\textit{\ and }$%
C(c,0),$ where $a>0,b>0$ and $\ c>0$, of the function 
\begin{equation*}
F(x,y)=\left( \frac{ax-by+ab}{\sqrt{a^{2}+b^{2}}}\right) ^{n}+\left( \frac{%
-ax-cy+ac}{\sqrt{a^{2}+c^{2}}}\right) ^{n}+y^{n},
\end{equation*}%
is 
\begin{equation*}
\frac{a^{n}}{\lambda ^{n}}(b+c)^{n}[t^{n}+r^{n}+1]
\end{equation*}%
attained at the point $(x_{\min },y_{\min })$ inside the triangle, where $%
x_{\min }=-\frac{bq-cpt}{\lambda }$ and $y_{\min }=\frac{abr+acr}{\lambda }.$
\end{theorem}

\begin{proof}
Define the Lagrangian function 
\begin{equation*}
L=F(x,y)-\lambda _{1}(ax-by+ab)-\lambda _{2}(-ax-cy+ac)-\lambda _{3}y,
\end{equation*}%
where $\lambda _{i}\geq 0,1\leq i\leq 3.$ Denote $\mathcal{D}=\left( \frac{%
ax-by+ab}{p}\right) ^{n-1}\geq 0$ and $\mathcal{E}=\left( \frac{-ax-cy+ac}{q}%
\right) ^{n-1}\geq 0.$ Then, the Karush-Kuhn-Tucker conditions (see \cite[%
p.341]{LY}) are%
\begin{eqnarray}
L_{x} &=&\frac{na}{p}\mathcal{D}-\frac{na}{q}\mathcal{E}-\lambda
_{1}a+\lambda _{2}a=0  \label{eq9} \\
L_{y} &=&\frac{-nb}{p}\mathcal{D}-\frac{nc}{q}\mathcal{E}+ny^{n-1}+\lambda
_{1}b+\lambda _{2}c-\lambda _{3}=0  \label{eq10} \\
\lambda _{1}(ax-by+ab) &=&0  \label{eq11} \\
\lambda _{2}(-ax-cy+ac) &=&0  \label{eq12} \\
\lambda _{3}y &=&0.  \label{eq13}
\end{eqnarray}%
We note first that $\lambda _{i}$ may be nonzero only if the corresponding
constraint is active. Hence, $(ax-by+ab)>0$ implies $\lambda _{1}=0$ and $%
\lambda _{1}>0$ implies $(ax-by+ab)=0.$ The same is true for $\lambda _{2}$
and $\lambda _{3}.$ Thus, to find a solution we define various combinations
of active constraints and check the signs of the resulting Lagrange
multipliers. Therefore, we have to consider the following cases derived from
equations (\ref{eq11})-(\ref{eq13}), according to whether $\lambda _{i}>0$
or $\lambda _{i}=0$. Notice also that the case $\lambda _{1}>0,\lambda
_{2}>0 $ and $\lambda _{3}>0$ implies an empty solution. Thus, we have to
consider $7$ cases.

\textsf{Case 1:}{\Large \ }$\lambda _{1}=\lambda _{2}=\lambda _{3}=0$ (The
feasible region is the set of interior points of $\Delta $).

In this case 
\begin{equation*}
L=F(x,y)=\left( \frac{ax-by+ab}{\sqrt{a^{2}+b^{2}}}\right) ^{n}+\left( \frac{%
-ax-cy+ac}{\sqrt{a^{2}+c^{2}}}\right) ^{n}+y^{n}.
\end{equation*}

Differentiating the function $F(x,y)$ with respect to $x$ and $y$, setting $%
\frac{\partial F}{\partial x}=\frac{\partial F}{\partial y}=0$ and putting $%
p=\sqrt{a^{2}+b^{2}},$ $q=\sqrt{a^{2}+c^{2}}$ we obtain the following
equations 
\begin{equation}
\frac{\partial F}{\partial x}=\frac{na}{p}\left( \frac{ax-by+ab}{p}\right)
^{n-1}-\frac{na}{q}\left( \frac{-ax-cy+ac}{q}\right) ^{n-1}=0,  \label{eq3}
\end{equation}%
\begin{equation}
\frac{\partial F}{\partial y}=\frac{-nb}{p}\left( \frac{ax-by+ab}{p}\right)
^{n-1}-\frac{nc}{q}\left( \frac{-ax-cy+ac}{q}\right) ^{n-1}+ny^{n-1}=0.
\label{eq4}
\end{equation}%
Since the point $(x,y)$ is located inside the triangle $\Delta $ then $%
ax-by+ab>0$, $-ax-cy+ac>0$ and $y>0.$ Hence, putting $t=\sqrt[n-1]{\frac{p}{q%
}}$ then equation (\ref{eq3}) simplifies into 
\begin{equation}
\frac{ax-by+ab}{p}=t\frac{-ax-cy+ac}{q}.  \label{eq5}
\end{equation}%
Substituting back in equation (\ref{eq4}) and eliminating $y^{n-1}$ we get 
\begin{equation*}
y^{n-1}=\left( \frac{-ax-cy+ac}{q}\right) ^{n-1}\frac{b+c}{q},
\end{equation*}%
which is equivalent to the following equation%
\begin{equation}
y=\frac{-ax-cy+ac}{q}r.  \label{eq6}
\end{equation}%
Simplifying equations (\ref{eq5}) and (\ref{eq6}) we obtain the system of
equations%
\begin{equation*}
\left\{ 
\begin{array}{c}
arx+(cr+q)y=acr \\ 
(aq+apt)x+(ptc-qb)y=ptac-qab%
\end{array}%
\right. .
\end{equation*}%
Now, reducing the augmented matrix of the system yields the row echelon form 
\begin{equation*}
\left( 
\begin{array}{ccc}
1 & 0 & -\frac{bq-cpt}{q+br+cr+pt} \\ 
0 & 1 & \frac{abr+acr}{q+br+cr+pt}%
\end{array}%
\right) .
\end{equation*}%
\ Hence, the solution is given by $x_{0}=-\frac{bq-cpt}{\lambda }$ and $%
y_{0}=\frac{abr+acr}{\lambda },$ where $\lambda =q+br+cr+pt.$

Carrying out the computations of the second derivatives, substituting $%
\mathcal{G}=\left( \frac{ax-by+ab}{p}\right) ^{n-2}$, $\mathcal{H}=\left( 
\frac{-ax-cy+ac}{q}\right) ^{n-2}$, and computing the determinant of the
Hessian matrix $\nabla ^{2}F(x,y)=\left( 
\begin{array}{cc}
\frac{\partial ^{2}F}{\partial x^{2}} & \frac{\partial ^{2}F}{\partial
x\partial y} \\ 
\frac{\partial ^{2}F}{\partial y\partial x} & \frac{\partial ^{2}F}{\partial
y^{2}}%
\end{array}%
\right) $ we obtain 
\begin{eqnarray*}
&&\left\vert \nabla ^{2}F(x,y)\right\vert \\
&=&\frac{n^{2}(n-1)^{2}}{p^{2}q^{2}y^{2}}\left( \mathcal{G}a^{2}q^{2}y^{n}+%
\mathcal{H}a^{2}p^{2}y^{n}+\mathcal{GH}a^{2}b^{2}y^{2}+\mathcal{GH}%
a^{2}c^{2}y^{2}+2\mathcal{GH}a^{2}bcy^{2}\right) .
\end{eqnarray*}%
Again, since the point $(x,y)$ is located inside the triangle $\Delta $ then 
$ax-by+ab>0$, $-ax-cy+ac>0$ and $y>0$, so we have $\mathcal{G}>0$ and $%
\mathcal{H}>0$. Hence, $\frac{\partial ^{2}F}{\partial x^{2}}=a^{2}\frac{%
n\left( n-1\right) }{p^{2}}\mathcal{G}+a^{2}\frac{n\left( n-1\right) }{q^{2}}%
\mathcal{H}>0$ and $\left\vert \nabla ^{2}F(x,y)\right\vert >0.$
Consequently, the Hessian matrix is positive definite and $(x_{0},y_{0})=(-%
\frac{bq-cpt}{\lambda },\frac{abr+acr}{\lambda })$ is the optimal solution
which minimizes $F(x,y)$.

In order to validate that the point ($x_{0},y_{0})$ is located inside the
triangle, it is enough to show that the point ($x_{0},y_{0})$ satisfies the
following constraints: $ax_{0}-by_{0}+ab>0,-ax_{0}-cy_{0}+ac>0$ and $%
y_{0}>0. $ The last constraint is clear so we proceed to check the first two
conditions. \ Recall that $\lambda =q+br+cr+pt$ then 
\begin{equation}
ax_{0}-by_{0}+ab=ap\frac{t}{\lambda }\left( b+c\right) >0  \label{eq7}
\end{equation}%
and\qquad \qquad \qquad\ 
\begin{equation}
-ax_{0}-cy_{0}+ac=a\frac{q}{\lambda }\left( b+c\right) >0.  \label{eq8}
\end{equation}

Finally, we compute the value of $F(x,y)$ at the optimal solution ($%
x_{0},y_{0}).$ Substituting the expressions from (\ref{eq7}) and (\ref{eq8})
we get%
\begin{eqnarray*}
F(x_{0},y_{0}) &=&\left( \frac{apt\left( b+c\right) }{\lambda p}\right)
^{n}+\left( \frac{aq\left( b+c\right) }{\lambda q}\right) ^{n}+(\frac{abr+acr%
}{\lambda })^{n} \\
&=&\frac{a^{n}}{\lambda ^{n}}(b+c)^{n}[t^{n}+1+r^{n}].
\end{eqnarray*}

\textsf{Case 2:}{\Large \ }$\lambda _{1}=\lambda _{2}=0${\Large \ }and%
{\Large \ }$\lambda _{3}>0$ (The feasible region is the segment $BC$).

From equations (\ref{eq13}) and (\ref{eq10}) we have $y=0$ and $\lambda _{3}=%
\frac{-nb}{p}\mathcal{D}-\frac{nc}{q}\mathcal{E}\leq 0,$ a contradiction.

\textsf{Case 3:}{\Large \ }$\lambda _{1}=\lambda _{3}=0${\Large \ }and%
{\Large \ }$\lambda _{2}>0$ (The feasible region is the segment $AC$).

From equation (\ref{eq12}) we have $-ax-cy+ac=0$ and therefore $\mathcal{E}%
=0.$ On the other hand, equation (\ref{eq9}) implies $\lambda _{2}=-\frac{n}{%
p}\mathcal{D}\leq 0,$ a contradiction.

\textsf{Case 4:}{\Large \ }$\lambda _{2}=\lambda _{3}=0${\Large \ }and%
{\Large \ }$\lambda _{1}>0$ (The feasible region is the segment $AB$).

From equation (\ref{eq11}) we have $ax-by+ab=0$ and therefore $\mathcal{D}%
=0. $ Equation (\ref{eq9}) implies $\lambda _{1}=-\frac{n}{q}\mathcal{E}\leq
0,$ a contradiction.

\textsf{Cases 5-7:}{\Large \ }Only one of $\lambda _{1},\lambda _{2}$\ or $%
\lambda _{3}$\ is $0$. The three distinct pairs of equations (\ref{eq11})-(%
\ref{eq13})\ imply that the feasible solution is one of the corresponding
vertices $A(0,a),B(-b,0)$\textit{\ or }$C(c,0).$ Hence, we have to compare
the values of the function $F$ at the point $(x_{0},y_{0})$ and at the
vertices of the triangle: $F(x_{0},y_{0})=\frac{a^{n}}{\lambda ^{n}}%
(b+c)^{n}[t^{n}+r^{n}+1]$, $F(0,a)=a^{n}$, $F(-b,0)=a^{n}(\frac{b+c}{q})^{n}$
and $F(c,0)=a^{n}(\frac{b+c}{p})^{n}.$

Assume first that $b\leq c.$ Now, from (\ref{eq2}) we have $t^{n}+r^{n}+1=t%
\frac{p}{q}+r\frac{b+c}{q}+1=\frac{\lambda }{q}.$ Observe that $\frac{q}{%
\lambda }<1$ and $\frac{p}{q}\leq 1.$ Hence, we can deduce that the value of 
$F$ at $(x_{0},y_{0})$ is less than each of its values at the points $%
(0,a),(-b,0)$ and $(c,0)$ (the details are left to the reader).
Consequently, for $b\leq c$, the minimum in the the closed triangle $\Delta $
of the function $F(x,y)$ is $\frac{a^{n}}{\lambda ^{n}}%
(b+c)^{n}[t^{n}+r^{n}+1]$ attained at the point $(x_{\min },y_{\min
})=(x_{0},y_{0})$ inside the triangle, where $x_{\min }=-\frac{bq-cpt}{%
\lambda }$ and $y_{\min }=\frac{abr+acr}{\lambda },$ in agreement with the
statement of the theorem.

Finally, the case $b>c$ can be proved by reflecting the triangle $\Delta $
across the $y-$axis. The set of values of the function $F$ will be preserved
and the minimum, inside the image triangle $\Delta ^{\prime },$ will be
attained at the reflection point $(x_{\min }^{\prime },y_{\min }^{\prime
})=(-x_{\min },y_{\min })=(-\frac{b^{\prime }q^{\prime }-c^{\prime
}p^{\prime }t^{\prime }}{\lambda ^{\prime }},\frac{a^{\prime }b^{\prime
}r^{\prime }+a^{\prime }c^{\prime }r^{\prime }}{\lambda ^{\prime }})$, where 
$b^{\prime }=c$, $c^{\prime }=b$, $p^{\prime }=q$, $q^{\prime }=p$, $%
r^{\prime }=\frac{r}{t}$, $t^{\prime }=\frac{1}{t}$ and $\lambda ^{\prime }=%
\frac{\lambda }{t}.$ We leave the details to the reader to show that $%
x_{\min }$ and $y_{\min }$ satisfy the same relations as in the theorem.
\end{proof}

\subsection*{The special case $n=2$}

If $n=2$ then the minimal sum of the squared distances from the sides of the
triangle $\Delta $, with vertices $A(0,a),B(-b,0)$\textit{\ }and\textit{\ }$%
C(c,0),$ is $\frac{a^{2}}{\lambda ^{2}}(b+c)^{2}[t^{2}+r^{2}+1].$ This
minimum is attained at the point $(x_{\min },y_{\min })=(-\frac{bq-cpt}{%
\lambda },\frac{abr+acr}{\lambda })$ inside the triangle, where $t=\frac{p}{q%
}$ and $r=\frac{b+c}{q}$. In particular, if the triangle $\Delta $ is
isosceles then $b=c.$ Therefore, the minimal sum of the squared distances is 
$\frac{a^{2}}{\lambda ^{2}}(b+c)^{2}[t^{2}+r^{2}+1]=\frac{2a^{2}b^{2}}{%
a^{2}+3b^{2}}.$ Moreover, $x_{\min }=-\frac{bq-cpt}{\lambda }=0$ and $%
y_{\min }=\frac{abr+acr}{\lambda }=\frac{2ab\frac{2b}{p}}{2\frac{a^{2}+3b^{2}%
}{p}}=\frac{2ab^{2}}{a^{2}+3b^{2}}.$

\subsection*{The sequence of critical points}

We fix the triangle $ABC$ with vertices $A(0,a),B(-b,0)$\textit{\ }and%
\textit{\ }$C(c,0)$ and for each $n\geq 2$ we let $(x_{\min ,n},y_{\min ,n})$
be the critical point at which the function $F$ in Eq. (\ref{eq1.4}) attains
its minimum. By Theorem \ref{mim inside triangle}, we have

\begin{equation*}
x_{\min ,n}=-\frac{bq-cpt_{n}}{\lambda _{n}}\text{ and }y_{\min ,n}=\frac{%
ab+ac}{\lambda _{n}}r_{n}.
\end{equation*}%
where by Eq. (\ref{eq2}), $t_{n}=\sqrt[n-1]{\frac{p}{q}}$, $r_{n}=\sqrt[n-1]{%
\frac{b+c}{q}}$, $\lambda _{n}=q+(b+c)r_{n}+pt_{n}$, $p=\sqrt{a^{2}+b^{2}}$%
and $q=\sqrt{a^{2}+c^{2}}.$   

Clearly, for fixed $a,b$ and $c$ we have $\underset{n\rightarrow \infty }{%
\lim }r_{n}=\underset{n\rightarrow \infty }{\lim }t_{n}=1$ and $\underset{%
n\rightarrow \infty }{\lim }\lambda _{n}=p+q+b+c.$ Hence,

\begin{equation*}
\underset{n\rightarrow \infty }{\lim }x_{\min ,n}=-\frac{bq-cp}{p+q+b+c},
\end{equation*}%
and

\begin{equation*}
\underset{n\rightarrow \infty }{\lim }y_{\min ,n}=\frac{ab+ac}{p+q+b+c}.
\end{equation*}%
Now,  
\begin{equation*}
E(-\frac{bq-cp}{p+q+b+c},\frac{ab+ac}{p+q+b+c})
\end{equation*}
is the incenter of the triangle $ABC$. To verify this fact it is sufficient
to prove that the point $E$ has equal distances from the sides of $ABC.$ By
Eq. (\ref{eq1}), the distance of $E$ from the side $AB$ is given by%
\begin{eqnarray*}
d_{1} &=&\left\vert \frac{-a\frac{bq-cp}{p+q+b+c}-b\frac{ab+ac}{p+q+b+c}+ab}{%
p}\right\vert  \\
&=&\allowbreak \frac{ab+ac}{b+c+p+q}.
\end{eqnarray*}%
Likewise, the distance of $E$ from the side $AC$ is given by

\begin{eqnarray*}
d_{2} &=&\left\vert \frac{a\frac{bq-cp}{p+q+b+c}-c\frac{ab+ac}{p+q+b+c}+ac}{q%
}\right\vert  \\
&=&\frac{ab+ac}{b+c+p+q}.
\end{eqnarray*}

Thus, we have the following conclusion.

\begin{corollary}
Let $ABC$ be the triangle with vertices $A(0,a),B(-b,0)$\textit{\ }and%
\textit{\ }$C(c,0)$ and let $(x_{\min ,n},y_{\min ,n})$ be the critical
point at which the function $F$ in Eq. (\ref{eq1.4}) attains its minimum.
Then

\begin{equation*}
\underset{n\rightarrow \infty }{\lim }x_{\min ,n}=-\frac{bq-cp}{p+q+b+c},
\end{equation*}%
and%
\begin{equation*}
\underset{n\rightarrow \infty }{\lim }y_{\min ,n}=\frac{ab+ac}{p+q+b+c},
\end{equation*}%
where $(-\frac{bq-cp}{p+q+b+c},\frac{ab+ac}{p+q+b+c})$ is the incenter of
the triangle $ABC.$
\end{corollary}

\bigskip

\end{document}